\documentclass[12pt]{article}
 \usepackage[utf8]{inputenc}

      \usepackage{latexsym}
         \usepackage[reqno, namelimits, sumlimits]{amsmath}
         \usepackage{amssymb, amsfonts}
         \usepackage{amsthm}
           
 \newtheorem{theorem}{Theorem}[section]

 \newtheorem{pro}[theorem]{Proposition}

\author{ Gregory Seregin
}

\title{On potential Type II blowups for the Navier-Stokes equations  }

\author{G.~Seregin\footnote{University of Oxford, Mathematical Institute, OxPDE, Oxford, UK and St Petersburg Department of Steklov Mathematical Institute, RAS, Russia, email address: \texttt{seregin@maths.ox.ac.uk}}
}

\date{}

\begin{document}

\maketitle

\rightline{\it In Memory of Hermann Sohr}

\begin{abstract} In the present note, certain scenarios of potential Type II blowups of solutions to the Navier-Stokes equations are considered on the local level. They generalise particular scenarios described  in the previous papers of the author. The main features of  the approach, adopted in the note, are a zoom based on the Euler scaling and Liouville type theorems for the Euler equations
in classes motived by a particular scenario of the Type II blowup.
\end{abstract}

{\bf Keywords} Navier-Stokes equations,
regularity, blowups.  

{\bf Data availability statement}
Data sharing not applicable to this article as no datasets were generated or analysed during the current study.

{\bf Acknowledgement} The work is supported by Leverhulme Emeritus Fellowship 2023.

\setcounter{equation}{0}
\section{Introduction}

The main purpose of this note is to further develop  the analysis of potential Type II blowups of suitable weak solutions to the Navier-Stokes equations, as initiated in papers \cite{Seregin2023}-\cite{Seregin2025}.
To this end, let us recall the main  notions and  notation adopted in the aforementioned papers. We begin  with the definition 
of suitable weak solutions to the Navier-Stokes  equations in the canonical space-time domain $Q$, 
as introduced in the seminal paper \cite{CKN} and  subsequently refined  in  papers  \cite{Lin, LS1999}. A    pair of functions $v$ and $q$ is called a suitable weak solution to the Navier-Stokes equations in space-time domain $Q$ if it has the following properties:
\begin{itemize}
	\item $v\in L_\infty(-1,0;L_2(B)),\quad \nabla v\in L_2(Q),\quad q\in L_\frac 32(Q);$
	\item the Navier-Stokes equations
$$\partial_tv+v\cdot\nabla v-\Delta v+\nabla q=0,\quad{\rm div}\,v=0$$
are satisfied in $Q$ in the sense of distributions;
\item
for a.a. $t\in ]-1,0[$, the local energy inequality 
$$\int\limits_B\varphi(x,t)|v(x,t)|^2dx+2\int\limits^t_{-1}\int\limits_B\varphi|\nabla v|^2dxd\tau\leq 
$$
$$\leq \int\limits^t_{-1}\int\limits_B (|v|^2(\partial_t\varphi+\Delta\varphi)+v\cdot\nabla\varphi(|v|^2+2q))dxd\tau$$ holds
for all smooth non-negative functions $\varphi$ vanishing in a vicinity of the parabolic boundary of the cylinder $Q$. 
\end{itemize}
Here and in what follows, the standard notation is used:
$Q(r)=B(r)\times ]-r^2,0[$ is a parabolic cylinder 
and $B(r)$ is a spatial ball of radius $r$ centred at the origin $x=0$, $B=B(1)$, and $Q=Q(1)$.

According to the partial regularity theory for suitable weak solutions to the Navier-Stokes equations
 developed in the above mentioned Caffarelli-Kohn-Nirenberg paper \cite{CKN}, there is a universal constant $\varepsilon_*>0$ such that if 
\begin{equation}
	\label{partial regularity}
	g_0(v)=\min\{\liminf_{r\to 0}A(v,r),\liminf_{r\to 0}E(v,r),\liminf_{r\to 0}C(v,r)\}<\varepsilon_*,
	\end{equation}
then $z=(x,t)=0$ is a regular point of $v$, i.e., $v\in L_\infty(Q(r_*))$ for some 
  radius $r_*\in ]0,1[$. Here, the following quantities  have been used:
$$E(v,r)=\frac 1{r}\int\limits_{Q(r)}|\nabla v|^2dz, \quad A(v,r)=\sup\limits_{-r^2<t<0}\frac 1{r}\int\limits_{B(r)}|v(x,t)|^2dx,$$
$$C(v,r)=\frac 1{r^2}\int\limits_{Q(r)}|v|^3dz.$$
It remains unknown  whether or not $z=0$ is a regular point of $v$ for any pair $v$ and $q$, which is a suitable weak solution to the Navier-Stokes equation in $Q$. 

Assuming that there exists a suitable weak solution $v$ and $q$ in $Q$ with a singularity (blowup) at the origin $z=0$, let us distinguish between two potentially possible cases: $z=0$ is a Type I blowup if $g_0(v)<\infty$ and $z=0$ is Type II blowup if $g_0(v)=\infty$.

As in the  paper \cite{Seregin2023}, we are going to study one of the potentially possible  scenarios of Type II blowups generated  by the following two conditions. In the first one, it is supposed that there are  a number $\varepsilon_0>0$, a function $g: ]0,1[\mapsto ]0,\infty[$,
and a sequence $r_k\downarrow0$ as $k\to\infty$ such that 
\begin{equation}
\label{g to 0}
g(r_k)\to0,\qquad k\to\infty
\end{equation}
and
\begin{equation}
	\label{TypeII}
	g(r_k)M^{s,l}_{\kappa}(v,r_k)\geq \varepsilon_0, 
\end{equation}
for all natural $k$, where
$$M^{s,l}_\kappa(v,a)= \frac 1{a^\kappa}\int\limits_{-a^2} ^0\Big(\int\limits_{B(a)}|v|^sdx\Big)^\frac lsdt,$$
with numbers $s>1$ and $l>1$ satisfying restrictions:
\begin{equation}
	\label{restriction on s and l}
	l>\kappa:=l\Big(\frac 3s+\frac 2l-1\Big)>0.
\end{equation}

In papers \cite{Seregin2023}-\cite{Seregin2025}, the function  $g$ has a polynomial character: $g(r)=r^{\kappa(1-m_0)}$ with $0<m_0<1$.
In order to demonstrate  what kind of potential Type II blowup can be described by \eqref{g to 0} and \eqref{TypeII}, let us consider the following simple example:
\begin{equation}
	\label{direct ineq}
	|v(x,t)|\leq \frac {c\ln^\gamma(e/\sqrt{-t})}{\sqrt{-t}}
	\end{equation} 
for any $-1<t<0$ and for some $\gamma>0$.
Let further
$$g(r)=\frac 1{\ln^{\gamma l}(e/r)}.
$$ for some $1<l<2$.
Then, simple calculations 
show
$$\sup\limits_{0<r<1}g(r)M^{s,l}_\kappa(v,r)<\infty.$$
If we assume further that:
\begin{equation}
	\label{inverse ineq}
	|v(x,t)|\geq \frac {c_1\ln^\gamma(e/\sqrt{-t})}{\sqrt{-t}}
	\end{equation}
for all $-1<t<1$. Then, similar calculations give us:
$$\lim\limits_{r\to0}g(r)M^{s,l}_\kappa(v,r)\geq \varepsilon_0 $$
for some positive $\varepsilon_0$.

The second assumption in our scenario of potential Type II blowup is as follows: 
\begin{equation}
		\label{M1}
		M_1=\sup\limits_{0<r<1}\{A_{f}(v,r)+E_f(v,r)+D_f(q,r)\}<\infty,
	\end{equation}
	where the quantities
$$E_f(v,r)=\frac {f(r)}{r}\int\limits_{Q(r)}|\nabla v|^2dz, \quad A_{f}(v,r)=\sup\limits_{-r^2<t<0}\frac {f^2(r)}{r}\int\limits_{B(r)}|v(x,t)|^2dx,$$
$$D_f(q,r)=\frac {f^2(r)}{r^{2}}\int\limits_{Q(r)}|q|^\frac 32dz 
$$
have been involved. A function $f:]0,1]\mapsto]0,1]$ is supposed to be monotonically increasing so that 
$$\lim\limits_{\lambda\to 0}f(\lambda)=0, \qquad f(1)=1.$$

Again in the aforementioned papers of the author, the function $f$ has a particular form: $f(r)=r^{1-m}$ with a given number  $0<m<1$.

The interesting question is about a relationship between functions  $f$ and $g$ that does not exclude  Type II blowup scenario described by assumptions \eqref{TypeII} and \eqref{M1}.

\setcounter{equation}{0}

\section{
Scenario \eqref{TypeII} and \eqref{M1}    }

To state the main result of the section, let us introduce auxiliary numbers $p(\eta)$ and $q(\eta)$ in the following way
\begin{equation}
	\label{numbers}
\frac{1}{	p(\eta)}=\frac \eta 6+\frac {3(1-\eta)} {10}, \quad \frac 1{q(\eta)}=\frac \eta 2+\frac {3(1-\eta)} {10},
\end{equation}
where $0\leq \eta\leq 1$ is a parameter. It is easy to check 
that 
$$1>\frac 3{p(\eta)}+\frac 2{q(\eta)}-1=\frac 12>0$$
for all $0\leq \eta \leq 1$.

As to function $f$, the following  assumption has been added: for any number $0<a<\infty$, there is a limit
\begin{equation}
	\label{add assump}
	F(a)= \liminf\limits_{\lambda\to0}F_\lambda(a)>0,
	\end{equation}
  where
$$F_\lambda(a)=\frac {f(\lambda a)}{f(\lambda)}$$
provided $0<\lambda<1$ and $a\lambda\leq 1$.
\begin{theorem}
	\label{LPS1}
	Let  a pair $v$ and $q$ be a suitable weak solution to the Navier-Stokes equations in $Q$. It is supposed that this pair obeys  condition \eqref{M1}.

Assume further that, for some numbers $s$ and $l$, satisfying inequalities (\ref{restriction on s and l}), the additional restrictions
\begin{equation}
	\label{numberscond}s <p(\eta),\quad l<q(\eta)\end{equation}
hold with some parameter $0\leq\eta \leq 1$.

Suppose also that
\begin{equation}
\label{f and g}	
{\Big(f(\lambda)\Big)^{\frac l2(1+\frac 3s)-\frac 32(1-\frac s{p(\eta)})\frac ls}}\Big(g(\lambda \sqrt{f(\lambda)})\Big)^{-1}\to \infty\end{equation} 
as $\lambda\to0$.
	
	 Then 
\begin{equation}
	\label{no-typeII}
	\lim\limits_{r\to 0}g(r)M^{s,l}_{\kappa}(v,r)=0.\end{equation}	
	In other words, scenario \eqref{TypeII} and \eqref{M1} 
of potential Type II blowup is impossible.	
	\end{theorem}	
	\begin{proof}
Suppose that statement \eqref{no-typeII} is wrong. Then, there are a positive number $\varepsilon_0$ and a sequence $r_k\to 0$ as $k\to\infty$ such that
\begin{equation}
	\label{non-statementLPS}
g(r_k)	M^{s,l}_{\kappa}(v,r_k)\geq \varepsilon_0>0
\end{equation}	
for all natural numbers $k$. 

Now, given function $f$, define  the following scaled variables and scaled functions 
\begin{equation}
	\label{Euler-scaling}
	v^{\lambda}(y,\tau)=\lambda f(\lambda) v(x,t),\qquad q^{\lambda}(y,\tau)=\lambda^{2} f^2(\lambda)q(x,t),\end{equation}
where 
$$x=\lambda y,\qquad t=\lambda^2f(\lambda)\tau.$$
The choice of $\lambda$ is as follows:
$$\lambda=\lambda_k, \quad r_k=\lambda_k\sqrt{f(\lambda_k)}.$$
Then, after the change of variables, we arrive at the important inequality: 
\begin{equation}
	\label{non-trivial-2}
 \int\limits_{-1}^0\Big(\int\limits_{B\Big(\sqrt{f(\lambda)}\Big)}|v^{\lambda}|^sdy\Big)^\frac lsd\tau=
 \lambda^{l-\frac {3l}s-2+\kappa}f^{l-1+\frac \kappa 2}(\lambda)
  M^{s,l}_{\kappa}(v,r_k)\geq %
	\end{equation}
$$\geq  {\Big(f(\lambda)\Big)^{\frac l2(1+\frac 3s)}}\Big(g(r_k)\Big)^{-1}\varepsilon_0.$$

Next, setting $r=a\lambda $  for $a<1/\lambda$, we find, as a result of the change of variables, the following estimate:
$$M_1\geq E_f(v,r)=\frac {f(r)}{r}\int\limits^0_{-r^2}\int\limits_{B(r)}|\nabla v|^2dxdt=$$
$$=\frac {f(r)}{r}\frac {\lambda^5f(\lambda)}{\lambda^4f^2(\lambda)}\int\limits^0_{-(\frac r\lambda)^2/f(\lambda)}
\int\limits_{B(\frac r\lambda)}|\nabla v^{\lambda}|^2dyds\geq$$$$\geq\frac {f(a\lambda)}{f(\lambda)}\frac {1}{a}\int\limits^0_{-a^2}\int\limits_{B(2a)}|\nabla v^{\lambda}|^2dyds=E_{F_\lambda}(v^{\lambda},a).$$ Arguing 
further in the same way, we find
\begin{equation}
	\label{M1scaling}
	M_1\geq \sup\limits_{0<a<1/\lambda}\{A_{F_\lambda}(v^{\lambda},a)+E_{F_\lambda}(v^{\lambda},a)+D_{F_\lambda}(q^{\lambda},a)\}.\end{equation}

Known multiplicative inequalities and arguments, taken from  the paper \cite{Seregin2023}, lead to the existence of subsequences
of $v^{\lambda_k}$ and $q^{\lambda_k,}$ such that:
\begin{itemize}
	\item $v^{\lambda_k}\to u$ in $L_{3\nu}(Q(a))$;
	\item  $v^{\lambda_k} {\stackrel{*}\rightharpoonup} u$ in $L_{2,\infty}(Q(a))$	\item $\nabla v^{\lambda_k}\rightharpoonup \nabla u$ in $L_2(Q(a))$;
	\end{itemize}
for all $a>0$ and for all $1\leq \nu<\frac {10}9$. Moreover, the limit functions $u$ and $p$ possess the properties listed below:
$$	M_1\geq \sup\limits_{a>0}\{A_{F}(u,a)+E_F(u,a)+D_F(p,a)\}$$
	and
$$\partial_tu+u\cdot\nabla u=-\nabla p,\qquad {\rm div}\,u=0$$
in $Q_-$.

Repeating arguments of paper \cite{Seregin2023} that essentially based  on application of H\"older inequality together with certain multiplicative inequalities, we obtain from \eqref{M1scaling} the bound $$\|v^{\lambda_k}\|_{ p(\eta), q(\eta),Q}\leq \|v^{\lambda_k}\|_{6,2,Q}^\eta\|v^{\lambda_k}\|_{10/3,Q}^ {1-\eta}\leq c_1<\infty.$$  It is  valid 
all natural number $k$.

 Making use of  H\"older inequality one more time, we find from \eqref{non-trivial-2} that:
$$\Big(f(\lambda_k)\Big)^{\frac l2(1+\frac 3s)}\Big(g(r_k)\Big)^{-1}\varepsilon_0\leq $$$$\leq c\Big(f(\lambda_k)\Big)^{\frac 32(1-\frac s{ p(\eta)})\frac ls}\|v^{\lambda_k}\|^l_{ p(\eta), q(\eta),B\Big (\sqrt{f(\lambda_k)}\Big)\times]-1,0[}\leq $$
$$\leq c \Big(f(\lambda_k)\Big)^{\frac 32(1-\frac s{ p(\eta)})\frac ls}c_1^l$$
and thus
$$\Big(f(\lambda_k)\Big)^{\frac l2(1+\frac 3s)-\frac 32(1-\frac s{p(\eta)})\frac ls}\Big(g(r_k)\Big)^{-1}\varepsilon_0\leq cc^l_1.$$
Passing to the limit as $k\to\infty$, one can conclude that number $\varepsilon_0$ must vanish. This is a contradiction.
\end{proof}

Consider an example of application of Theorem \ref{LPS1}. To this end, let us introduce the functions
\begin{equation}\label{example}
	f(\lambda)=\frac {1}{\ln^\gamma{(e/\lambda)}}
\end{equation} 
 and 
$$g(r)=\frac {1}{\ln^\nu{(e/r)}}$$
for some positive parameters $\gamma$ and $\nu$.
Direct calculations show that
if 
$$\nu>\gamma\frac l2\Big(1+\frac 3{p(\nu)}\Big),$$
then 
$$\lim\limits_{r\to0}g(r)M^{s,l}_\kappa(v,r)=0$$
and moreover
$$\lim\limits_{r\to0}r^{\kappa(1-m_0)}M^{s,l}_\kappa(v,r)=0$$
for any $0<m_0<1$.

\setcounter{equation}{0}

\section{New Version of Condition
\eqref{TypeII}   }
Here, we replace assumption \eqref{TypeII}  with  another one: there are a positive number $\varepsilon_0$ and a sequence of $r_k\downarrow 0$ as $k\to\infty$ such that 
\begin{equation}	\label{one more}
g(r_k)\overline M^{s,l}_\kappa(v,r_k)\geq \varepsilon_0
\end{equation}
for all $k=1,2,,,$ where  
$$\overline M^{s,l}_\kappa(v,r)=\frac 1{r^\kappa}\int\limits^0_{-r^2f(r)}\Big(\int\limits_{B(r)}|v|^sdx\Big)^\frac lsdt $$ and
$$g(r)=(f(r))^{l-1}r^{l-2-3\frac ls+\kappa}=(f(r))^{l-1}.$$
Obviously, since
\begin{equation}
	\label{imply Type II}
M^{s,l}_{\kappa}(v,r)\geq \overline{M}^{s,l}_{\kappa}(v,r),\end{equation}
condition \eqref{one more} describes  a potential Type II blowup. 
Necessary condition of the fact that such a  Type II blowup might happen follows from \eqref{f and g}, see Theorem \ref{LPS1}: 
\begin{equation}
	\label{new condition}
\limsup\limits_{\lambda\to 0}\frac {(f(\lambda))^{\frac l2(1+\frac3{p(\eta)})}}{(f(\lambda\sqrt{f(\lambda)}))^{l-1}}<\infty.
\end{equation}

The scaling used in the proof of Theorem \ref{LPS1} leads to the following important identity
\begin{equation}
	\label{Non-trivial-identity}
	g(r_k)\overline M^{s,l}_\kappa(v,r_k)	=
	\int^0_{-1}\Big(\int\limits_B|v^{\lambda}|^sdy\Big)^\frac ls d\tau	\geq \varepsilon_0>0	\end{equation}
with $\lambda=\lambda_k=r_k$. Then, repeating the same arguments as in the proof of Proposition 1.2 in \cite{Seregin2023}, we prove  the following theorem.
\begin{theorem}
	\label{euler}
	Suppose that a pair $v$ and $q$ is a suitable weak solution to the Navier-Stokes equations in the unit space-time cylinder $Q$. Assume $v$ and $q$ satisfy the conditions \eqref{M1}, \eqref{one more}, and \eqref{new condition}.

Then, there are two functions $u$ and $p$ defined in $Q_-=\mathbb R^3\times]-\infty,0[$, with the following properties:
	$$\sup\limits_{a>0}\Big[\sup\limits_{-a^2<\tau<0}\frac {F^2(a)}{a}\int\limits_{B(a)}|u(y,\tau)|^2dy+\frac {F^2(a)}{a^{2}}\int\limits_{Q(a)}|p|^\frac 32dyd\tau+$$
 \begin{equation}
	\label{basicestimates}
+\frac {F(a)}{a}\int\limits_{Q(a)}|\nabla u|^2dyd\tau\Big]\leq c<\infty;\end{equation}
\begin{equation}
	\label{Euler}\partial_\tau u+u\cdot\nabla u+\nabla p=0, \quad{\rm div}\,u=0
\end{equation}in $Q_-=\mathbb R^3\times ]-\infty,0[$ in the sense of distributions;
  
 for a.a. $\tau_0\in ]-\infty,0[$, the local energy inequality
$$\int\limits_{\mathbb R^3}|u(y,\tau_0)|^2\varphi(y,\tau_0)dy\leq$$
\begin{equation}
	\label{enerylocal}
\leq \int\limits_{-\infty}^{\tau_0}\int\limits_{\mathbb R^3} \Big(|u|^2\partial_\tau\varphi+u\cdot\nabla\varphi(|u|^2+2
p)
\Big)dyd
\tau
\end{equation}
holds for non-negative $\varphi\in C^\infty_0(\mathbb R^3\times \mathbb R)$;
 
 the function $u$ is non-trivial in the sense
 \begin{equation}
	\label{nontrivial}
M^{s,l}_{\kappa}(u,1)
\geq \varepsilon_0/2.	
\end{equation} 
\end{theorem}

Coming back to our example \eqref{example}, we notice that
$$g(r)=\frac 1{(\ln(e/r))^{\gamma(l-1)}}.$$
Then verification of \eqref{new condition} can be reduced to the following inequality
$$l-1\leq \frac l2\Big(1+\frac 3{p(\eta)}\Big)$$
which is equivalent to 
$$\frac 12\leq \frac 1l+\frac 34-\frac{1}{q(\eta)}\leq \frac 34.$$
The latter means that there are no additional restrictions on $s$ and $l$. By the way,  in the above example:
$$F(a)=1.$$

Now, we would like to consider a more general example
\begin{equation}
	\label{general case}
	f(\lambda)=\frac {\lambda^{\alpha-1}}{\ln^\gamma(e/\lambda)}
\end{equation}
where $0<\lambda\leq 1$ and parameters $\alpha$ and $\gamma$ obey the restrictions:
\begin{equation}
	\label{restrictions}
	1<\alpha<2, \qquad 0<\gamma^2<\infty.
\end{equation}
In this case, 
$$F(a)=a^{\alpha-1}$$
and thus
$$\frac {F^2(a)}a=a^{2\alpha-3}.$$
From \eqref{basicestimates}, it follows that $u=0$ if $2\alpha-3>0$ and singularity \eqref{one more} cannot happen according to  Theorem \ref{euler}.
Hence,  again, we have to restrict ourselves to the case  
\begin{equation}
	\label{add restriction}
	2\alpha-3\leq 0
\end{equation}

So, now, the new condition \eqref{new condition} can be reduced to consideration of the following function
$$A(\lambda):=\lambda^{\kappa_1}\Big[\ln^\gamma(e/\lambda)\Big]^{\frac 12(l-1)(\alpha-1)-\frac l2(1+\frac3{p(\eta)})}\Big[\ln^\gamma(e/(\lambda\sqrt{f(\lambda)}))\Big]^{l-1}
$$
and its behaviour as $\lambda\to 0$. Here,
$$\kappa_1=(\alpha-1)\Big[\frac l2\Big (1+\frac 3{p(\eta)}\Big )-\frac {\alpha+1}2(l-1)\Big].
$$

If $\kappa_1<0$, then singularities of 
 type \eqref{one more} do not exist since
$$0\leftarrow g(r)M^{s,l}_\kappa(v,r)\geq g(r)\overline M^{s,l}_\kappa(v,r)\rightarrow 0
$$
as $r\to0$.

If $\kappa_1>0$, then singularities \eqref{one more} cannot be excluded and Theorem \ref{euler} can be applied.

Let us show that, for our values of parameters $\alpha$, $l$, and $s$, the second case takes place only. Indeed, by properties numbers $s<p(\eta)$ and $l<q(\eta)$, we find, see \eqref{add restriction}, 
$$
\kappa_1=(\alpha-1) \Big[\frac l2\Big (1+\frac 3{p(\eta)}+\frac 2{q(\eta)}-1+1- \frac 2{q(\eta)}\Big )-\frac {\alpha+1}2(l-1)\Big] =$$
$$=(\alpha-1)\Big[l\frac {3-2\alpha}4+\frac {\alpha+1}2-\frac l{q(\eta)}\Big]>0.
$$

\setcounter{equation}{0}

\section{The case with axial symmetry}

In this section, it is supposed that the pair $v$	and $q$ is an axisymmetric suitable weak solution to the Navier-Stokes equations in $Q$. It is supposed that axis $x_3$ is the axis of symmetry and $v=v_\varrho e_\varrho+v_\vartheta e_\vartheta+v_3e_3 $ and $v_{\varrho,\vartheta}=v_{\vartheta,\vartheta}=v_{3,\vartheta}=q_{,\vartheta}=0$, 
where $\varrho$, $\vartheta$, and $x_3$ are standard cylindrical coordinates.

Here, as in the paper \cite{Seregin2023-1}, we can make a plausible assumption that,  in addition, 
\begin{equation}
	\label{additional restriction}
	\sup\limits_{0<r<1}\{N^{s_1,l_1}_f(v,r) +A_{f}(v,r)+E_f(v,r)+D_f(q,r)\}= M_1<\infty
\end{equation}
where
$$N^{s_1,l_1}_f(v,r)=\frac {f^{l_1-1}(r)}{r^{l_1}}\int\limits^0_{-r^2}\Big(\int\limits_{B(r)}(|\nabla^2v|+|\partial_tv|)^{s_1}dx\Big)^{\frac {l_1}{s_1}}dt
$$
and numbers $1<l_1,s_1<\infty$ obey the restriction
$$\frac 3{s_1}+\frac 2{l_1}=4.
$$

Now, one can apply Theorem \ref{euler}, replacing in it condition 
\eqref{M1} with condition 
\eqref{additional restriction}. As a result, if the origin is a singular point of $v$ then there is a pair of functions $u$  and $p$ that satisfy the following properties:
\begin{itemize}
	\item 
$$\sup\limits_{a>0}\Big[\sup\limits_{-a^2<\tau<0}\frac {F^2(a)}{a}\int\limits_{B(a)}|u(y,\tau)|^2dy+\frac {F^2(a)}{a^{2}}\int\limits_{Q(a)}|p|^\frac 32dyd\tau+$$
 \begin{equation}
	\label{basicestimates axi}
+\frac {F(a)}{a}\int\limits_{Q(a)}|\nabla u|^2dyd\tau+\end{equation}$$+
\frac {F^{l_1-1}(a)}{a^{l_1}}\int\limits^0_{-a^2}\Big(\int\limits_{B(a)}(|\nabla^2u|+|\partial_\tau u|)^{s_1}dy\Big)^\frac {l_1}{s_1}d\tau\Big]\leq c<\infty;$$
\item
Euler equations \eqref{Euler}; 
\item 
local energy inequality \eqref{enerylocal}; 
\item
 non-triviality condition \eqref{nontrivial}.
\end{itemize}
 In addition, $u$ and $p$ are axially symmetric with the same axis of symmetric and swirl-free,
see  \cite{Seregin2023-1}.

As in the paper \cite{Seregin2023-1}, it is assumed 
that 
$$l_1\leq s_1
$$
and
$$|v(x,t)|\leq \frac c {|x'|f(|x'|)}
$$ for all $x\in B$. Then, using the same arguments (again as in \cite{Seregin2023-1}), we can prove that 
\begin{equation}
	\label{identity}
\int\limits_{Q_-}F(f)(\partial_t\varphi+u\cdot\nabla\varphi)dz=0
\end{equation}
for all test functions $\varphi(|x'|,x_3,t)\in C^\infty_0(Q_-)$, where
$$f=\frac{\omega_\vartheta(u)}\varrho,\qquad F(f)=\Phi(|f|),\qquad \Phi(q)=\frac 2{l_1}q^\frac {l_1}2,
$$ the function $\Phi$ defined for all $q\geq 0$ and 
$$
\omega_\theta(u)=
\omega (u)\cdot e_\theta={\rm rot} \,u\cdot e_\theta =u_{\varrho,3}-u_{3,\varrho}.
$$

From \eqref{identity}, one can derive a conservation law under the following additional condition
\begin{equation}
	\label{additional}
	\frac {a^{2-\frac 32l_1}}{F^\frac {l_1+1}2(a)}\to 0
\end{equation}
 as $a\to \infty$. Indeed, 
 choosing  a cut-off function $\varphi$ in \eqref{identity}, depending on a positive parameter $a$ in a suitable way, and passing to the limit as $a\to\infty$, see similar situation in paper \cite{Seregin2023-1}, we can show that, due to the assumption \eqref{additional}, the integral, containing spatial derivatives of $\varphi$, tends to zero. It leads
 to the conservation law of the form 
\begin{pro}
	\label{minorpropo}
	Assume that  all conditions of the section  hold and, in addition, there exists $t_0\leq 0$ such that 
	\begin{equation}
		\label{assumption}
		g(t_0):=\frac 2{l_1}\int\limits_{\mathbb R^3}\Big(\frac {|\omega_\vartheta(u(x,t_0)|}{r}\Big)^\frac {l_1}2dx<\infty.
	\end{equation}
Then 
\begin{equation}
	\label{preservation}
	g(t):=\frac 2{l_1}\int\limits_{\mathbb R^3}\Big(\frac {|\omega_\vartheta(u(x,t)|}{r}\Big)^\frac {l_1}2dx=g(t_0)
\end{equation}	
for all $t\leq 0$.
\end{pro}
The latter statement could be used for further analysis of ancient solutions to the Euler equations.

\end{document}